\documentclass[12pt]{amsart}

\usepackage{amsmath,amsfonts,amssymb,amsthm,graphicx,overpic,hyperref,cleveref}
\usepackage{float,enumitem}
\usepackage{color,xcolor}
\usepackage{tikz}

\usepackage[margin=1.5in]{geometry}  

\numberwithin{equation}{section}

\newtheorem{thm}{Theorem}[section]
\newtheorem{prop}[thm]{Proposition}
\newtheorem{prp}[thm]{Proposition}
\newtheorem{cor}[thm]{Corollary}
\newtheorem{lem}[thm]{Lemma}

\theoremstyle{definition}

\newtheorem{defn}[thm]{Definition}

\theoremstyle{remark}
\newtheorem{rem}[thm]{Remark}

\newcommand{\nc}{\newcommand}
\nc{\dmo}{\DeclareMathOperator}
\usepackage{dsfont}

\nc{\abs}[1]{\left| #1 \right|}
\nc{\bigO}[1]{O\left(#1\right)}
\nc{\card}[1]{\left|#1\right|}
\nc{\ceil}[1]{\left\lceil #1 \right\rceil}
\nc{\CC}{\mathbb{C}}
\nc{\dilog}{\mathcal{L}}
\nc{\floor}[1]{\left\lfloor #1 \right\rfloor}
\nc{\ind}{\mathds{1}}
\nc{\ZZ}{\mathbb{Z}}
\nc{\len}[1]{\left| #1 \right|}
\dmo{\EL}{EL}
\nc{\littleo}[1]{o\left(#1\right)}
\dmo{\Mat}{Mat}
\nc{\norm}[1]{\left|\left| #1 \right|\right|}
\nc{\PP}{\mathbb{P}}
\nc{\NN}{\mathbb{N}}
\nc{\QQ}{\mathbb{Q}}
\nc{\RR}{\mathbb{R}}
\nc{\MF}{\mathcal{MF}}
\nc{\PMF}{\mathcal{PMF}}

\nc{\Bclose}{\overline{\mathcal{B}}}
\nc{\Bussp}{\mathcal{B}}

\nc{\st}[2]{\left\{\, #1 \,:\, #2\,\right\}}
\dmo{\supp}{supp}
\nc{\tr}[1]{\mathrm{tr}\left(#1\right)}
\nc{\what}{\widehat}
\dmo{\im}{Im}
\dmo{\re}{Re}
\nc{\eps}{\varepsilon}
\dmo{\li}{li}
\dmo{\chr}{chr}

\dmo{\arccosh}{arccosh}
\dmo{\arcsinh}{arcsinh}
\dmo{\area}{area}
\dmo{\conv}{conv}
\dmo{\diam}{diam}
\dmo{\DD}{\mathbb{D}}
\dmo{\dist}{\mathrm{d}}
\nc{\HH}{\mathbb{H}}
\dmo{\Isom}{Isom}
\dmo{\MCG}{MCG}
\dmo{\Mod}{\mathcal{M}}
\nc{\Sphere}{\mathbb{S}}
\dmo{\Teich}{\mathcal{T}}
\nc{\teichmuller}{Teichmüller }
\nc{\Torus}{\mathbb{T}}
\dmo{\vol}{vol}
\dmo{\WP}{WP}
\nc{\tray}[1]{R(#1)}
\nc{\trayt}[2]{R_{#1}(#2)}

\dmo{\Aut}{Aut}
\dmo{\Fix}{Fix}
\dmo{\GL}{GL}
\dmo{\id}{id}
\dmo{\PSL}{PSL}
\dmo{\PGL}{PGL}
\dmo{\Rep}{Rep}
\dmo{\SL}{SL}
\dmo{\SO}{SO}
\dmo{\sym}{\mathfrak{S}}

\nc{\calA}{\mathcal{A}}
\nc{\calB}{\mathcal{B}}
\nc{\calC}{\mathcal{C}}
\nc{\calE}{\mathcal{E}}
\nc{\calF}{\mathcal{F}}
\nc{\calD}{\mathcal{D}}
\nc{\calG}{\mathcal{G}}
\nc{\calH}{\mathcal{H}}
\nc{\calI}{\mathcal{I}}
\nc{\calJ}{\mathcal{J}}
\nc{\calK}{\mathcal{K}}
\nc{\calL}{\mathcal{L}}
\nc{\calM}{\mathcal{M}}
\nc{\calN}{\mathcal{N}}
\nc{\calO}{\mathcal{O}}
\nc{\calP}{\mathcal{P}}
\nc{\calQ}{\mathcal{Q}}
\nc{\calR}{\mathcal{R}}
\nc{\calS}{\mathcal{S}}
\nc{\calT}{\mathcal{T}}
\nc{\calU}{\mathcal{U}}
\nc{\calV}{\mathcal{V}}
\nc{\calW}{\mathcal{W}}
\nc{\calX}{\mathcal{X}}
\nc{\calY}{\mathcal{Y}}
\nc{\calZ}{\mathcal{Z}}

\nc{\cech}{\v{C}ech}

\title{Busemann points are nowhere dense}

\author{Aitor Azemar}
\address{}
\email{aitor@azemar.xyz}

\author{Maxime Fortier Bourque}
\address{D\'epartement de math\'ematiques et de statistique, Universit\'e de Montr\'eal, 2920, chemin de la Tour, Montr\'eal (QC), H3T 1J4, Canada}
\email{maxime.fortier.bourque@umontreal.ca}

\date{\today}

\AtBeginDocument{%
   \def\MR#1{}
}
\begin{document}
\begin{abstract}
We prove that the set of Busemann points (the limits of almost-geodesic rays) is nowhere dense in the horoboundary of the Teichm\"uller metric for all Teichm\"uller spaces of complex dimension strictly larger than $1$. This shows that the Teichm\"uller metric is far from having non-positive curvature in a certain sense.
\end{abstract}

\maketitle

\section{Introduction}

One may embed any metric space $(X,d)$ into the vector space $C(X)/\RR$ of continuous real-valued functions on $X$ (equipped with the compact-open topology) modulo constants by sending any $x\in X$ to the equivalence class of the distance function $y \mapsto d(y,x)$. If $(X,d)$ is \emph{proper}, meaning that its closed balls are compact, then the closure of the image of this map is compact and called the \emph{horofunction compactification} of $(X,d)$ (see \cite[Section 1.2]{Gromov} and \cite[Section 4]{Rieffel}). Its boundary points are called \emph{horofunctions} and the set of all horofunctions is the \emph{horoboundary}. \emph{Busemann points} are special horofunctions obtained as limits of \emph{almost-geodesic rays}, a generalization of geodesic rays \cite[Section 4]{Rieffel}. 

The horofunction compactification, horoboundary, or Busemann points, have been calculated explicitly for a few families of metric spaces such as for $\mathrm{CAT}(0)$ spaces  (where all horofunctions are Busemann \cite[Theorem 8.13]{BridsonHaefliger}), Hilbert geometries on convex bodies \cite{WalshHilbert}, finite-dimensional normed vector spaces \cite{KarlssonEtAl,WalshMinkowski}, the continuous Heisenberg group equipped with either the Kor\'anyi \cite{KleinNicas1} or Carnot--Carath\'eodory metric \cite{KleinNicas2}, and Teichm\"uller space equip\-ped with the Thurston metric \cite{WalshThurston}.

For the Teichm\"uller space of a surface of finite type equipped with the Teichm\"uller me\-tric, Liu and Su \cite{LiuSu} proved that the horofunction compactification is isomorphic to the Gardiner--Masur compactification \cite{gardmasur1991extrgeom} obtained by replacing hyperbolic length with (the square root of) extremal length in Thurston's compac\-ti\-fi\-cation. In \cite{Miyachi}, Miyachi proved the existence of non-Busemann points in the closure of the set of Busemann points. In addition to re\-de\-ri\-ving these two results in \cite{walsh2019asympgeom}, Walsh proved that all Busemann points are limits of geodesic rays \cite[Theorem 3]{walsh2019asympgeom} and obtained explicit formulas for them in the Gardiner--Masur compactification \cite[Corollary 1]{walsh2019asympgeom}. Azemar used these formulas together with some horofunctions constructed in \cite{fortier2019Divergent} to show that Busemann points are not dense in the horo\-boun\-dary \cite[Theorem 1.8]{azemar2021qualitative} unless the Teichm\"uller space is isometric to the hyperbolic plane.

Here we strengthen these results of Miyachi and Azemar by showing that Busemann points are nowhere dense in the horoboundary of the \teichmuller metric (for Teichm\"uller spaces of complex dimension larger than $1$).

\begin{thm}\label{th:maintheorem}
Let $g$ and $p$ be non-negative integers such that $3g+p>4$ and let $\calT_{g,p}$ be the Teichm\"uller space of Riemann surfaces of genus $g$ with $p$ punctures. Then the set of Busemann points is nowhere dense in the horo\-boun\-dary of the Teichmüller metric on $\calT_{g,p}$.
\end{thm}

This means that the closure of the set of Busemann points has empty interior as a subset of the horoboundary. Equivalently, the complement of the closure of the set of Busemann points is dense in the horoboundary. Since the horofunctions of a $\mathrm{CAT}(0)$ space are all Busemann points, \Cref{th:maintheorem} shows that the Teichm\"uller metric is rather far from having non-positive curvature, adding to several previous results in that direction \cite{Masur,MasurWolf,Minsky,FBRafi}.

By \cite{KleinNicas1} and \cite{KleinNicas2}, the horoboundary of the continuous Heisenberg group of dimension $(2n+1)$ with respect to the Kor\'anyi or Carnot--Carath\'eo\-do\-ry metric is homeomorphic to a closed disk of dimension $2n$, with the set of Busemann points given as the $(2n-1)$-dimensional spherical boundary of that disk. In particular, Busemann points are also nowhere dense for continuous Heisenberg groups, but they still form closed subsets while this is not the case for the Teichm\"uller metric by Miyachi's result. Furthermore, the horoboundary of the Teichm\"uller metric appears to be more complicated topologically.

Indeed, recall that when when $2g+p > 2$, the Teichm\"uller space $\calT_{g,p}$ is homeomorphic to $\RR^{6g+2p-6}$. Its horoboundary containts a sphere of dimension $6g+2p-7$ which can be identified with Thurston's sphere of projective measured foliations \cite[Theorem 7.1]{gardmasur1991extrgeom} and onto which the horoboundary admits a retraction \cite[Theorems 1.1 and 8.1]{azemar2021qualitative}. One can therefore think of the horo\-boun\-dary of $\calT_{g,p}$ as a $(6g+2p-7)$-dimensional sphere with spikes (the fibers of the retraction) attached to it. The Thurston sphere is in some sense on the outside of the horocompactification \cite[Proposition 6.3 and Corollary 4.8]{azemar2021qualitative} and the Busemann points are the tips of the inward spikes \cite[Proposition 3.9]{azemar2021qualitative}. Note that Busemann points are dense within the Thurston sphere (this follows from the proof of \cite[Theorem 7.1]{gardmasur1991extrgeom}), but the Thurston sphere itself is nowhere dense within the horoboundary \cite[Theorem 7.5]{azemar2021qualitative}. These results and \Cref{th:maintheorem} suggest that $\calT_{g,p}$ sits in a rather convoluted way inside its horofunction compactification, which for all we know could still be homeomorphic to a closed ball (see \cite[Figure 1]{azemar2021qualitative} for a sketch of what that might look like).

The paper is organized as follows. \Cref{sec:defs} gives necessary background in Teichm\"uller theory and sets our notation. \Cref{sec:points_outside} recalls a strategy from \cite{azemar2021qualitative} for constructing boundary points outside the closure of Busemann points. \Cref{sec:simplices} then shows how to construct a top-dimensional simplex of points outside the closure of Busemann points by choosing a suitable branched cover of our surface onto the five-times-punctured sphere. This simplex $\Delta_X$ depends on a choice of basepoint $X$ and we show that by pushing $X$ to infinity along certain Teichm\"uller rays, $\Delta_X$ accumulates onto a top-dimensional simplex $\Delta$ of Busemann points corresponding to Jenkins-Strebel quadratic differentials. Lastly, \Cref{sec:orbits} shows that the mapping class group orbit of any such top-dimensional Jenkins-Strebel simplex is dense among Busemann points and uses this to prove \Cref{th:maintheorem}. Note that the action of the mapping class group on the horoboundary is not minimal (for instance, the Thurston sphere is a closed invariant set), hence taking the orbit of a single point outside the closure of Busemann point would not be enough to show that the complement of the closure is dense.

\section{Definitions and notation} \label{sec:defs}

We recall some standard definitions and set our notation here. We refer the reader to \cite{Hubbard,FarbMargalit,FLP} for further background on Teichm\"uller theory.

\subsection{Teichm\"uller space and the Teichm\"uller metric}

Let $S$ be a closed orientable surface with finitely many points removed (perhaps none). We say that $S$ is a surface of \emph{finite type}.

The \emph{Teichm\"uller space} $\Teich(S)$ is the set of equivalence classes of pairs $(X,f)$ where $X$ is a compact Riemann surface minus a finite number of points and $f:S\to 
 X$ is a \emph{marking}, that is, an orientation-preserving homeomorphism. Two such pairs $(X,f)$ and $(Y,g)$ are \emph{equivalent} if $g\circ f^{-1}$ is homotopic to a biholomorphism. We typically keep the marking implicit and simply write $X \in \Teich(S)$ instead of $[(X,f)]\in \Teich(S)$. The \emph{mapping class group} $\MCG(S)$ of homotopy classes of orientation-preserving homeomorphisms $h:S\to S$ acts on $\Teich(S)$ by $[h]\cdot [(X,f)]:=[(X,f\circ h^{-1})]$, which we denote by $h(X)$ when the marking is implicit.

The \emph{Teichm\"uller distance} between two points $[(X,f)],[(Y,g)]\in \Teich(S)$ is defined as
\[
d(X,Y):= \inf \frac12 \log K,
\]
where the infimum is over the set of $K\geq 1$ such that there exists a $K$-quasiconformal homeomorphism homotopic to $g\circ f^{-1}$.

\subsection{Quadratic differentials and Teichm\"uller geodesics}

A \emph{quadratic differential} on a Riemann surface $X$ is a function $q: TX \to \CC$ such that $q(\lambda v) = \lambda^2 q(v)$ for every $v \in TX$ and every $\lambda>0$. We only consider quadratic differentials that are \emph{integrable}, meaning that \[\|q\|:= \int_X |q| < \infty,\] and \emph{holomorphic}, meaning that for any holomorphic vector field $v$ defined on an open set in $X$ , the function $q\circ v$ is holomorphic. The set of all integrable holomorphic quadratic differentials on $X$ will be denoted $Q(X)$ and its subset of quadratic differentials of norm $1$ will be denoted $Q^1(X)$.

A \emph{natural coordinate} for a (non-identically zero) quadratic differential $q$ is a chart $z:U \to \CC$ from an open set $U \subset X$ such that $q = dz^2$ on $TU$. Natural coordinates form a complex atlas on $X$ minus the finite set of zeros of $q$ whose transition functions are \emph{half-translations} of the form $z \mapsto \pm z+c$. One may deform this atlas by post-composing each natural coordinate with the matrix $\begin{pmatrix}e^t & 0 \\ 0 & e^{-t}\end{pmatrix}$ for $t\in \RR$. By Teichm\"uller's theorem, this path defines a geodesic with respect to the Teichm\"uller metric, which we denote by $R_q(t)$.

\subsection{Measured foliations and extremal length}

The space of measured foliations a surface of finite-type $S$, up to measure-equivalence, is denoted by $\MF(S)$ and $\PMF(S)=(\MF(S)\setminus\{0\})/\RR_{>0}$. For example, for a non-trivial quadratic differential $q$ on a Riemann surface $X$, there is an associated \emph{vertival foliation} $V(q)$ whose leaves map to vertical lines in the plane under natural coordinates, equipped with the transverse measure $\left|\re \sqrt q \right|$. The horizontal foliation $H(q)$ is defined analogously. We also define $H(0)=V(0)$ to be the zero or empty measured foliation. If $[(X,f)]\in \Teich(S)$, then we will abuse notation and consider $H(q)$ and $V(q)$ as measured foliations on $S$ (by pulling them back by $f$).

A theorem of Hubbard and Masur \cite{HubbardMasur} states that for any $X\in \Teich(S)$, the map
\[
\begin{array}{ccc}
Q(X) &\to & \MF(S) \\
 q &\mapsto & V(q)
 \end{array}
\]
is a homeomorphism, and similarly for the horizontal foliation map $V$. Since $V(\lambda^2 q) = \lambda V(q)$ for $\lambda > 0$, these maps restrict to a homeomorphisms between $Q^1(X)$ and $\PMF(S)$.

In particular, for every $F \in \MF(S)$ and every $X \in \Teich(S)$, there is a unique $q_F \in Q(X)$ such that $V(q_F) = F$. One may use this to define the \emph{extremal length} of $F$ on $X$ as
\[
\EL(F,X)= \| q_F \|,
\]
which agrees with the usual definition when $F$ is a simple closed curve.

\subsection{The Gardiner--Masur boundary}

 Let 
 \[
 \PP\left(\RR_{\geq 0}^{\MF(S)}\right) = \left( \RR_{\geq0}^{\MF(S)} \setminus \{ 0 \}\right) / \RR_{>0}
 \]
 be the space of non-trivial functions $\MF(S)\to \RR_{\geq 0}$ up to multiplication by strictly positive scalars.

By work of Gardiner and Masur \cite[Theorem 6.1]{gardmasur1991extrgeom}, the map
 \[
 \begin{array}{ccl}
\Teich(S) & \to & \PP\left(\RR_{\geq 0}^{\MF(S)}\right) \\
 X &\mapsto & \left[\sqrt{\EL(F,X)}\right]_{F \in \MF(S)}
 \end{array}
 \]
 is an embedding, and the closure of its image is compact, hence defines a compactification called the \emph{Gardiner--Masur compactification}  of $\Teich(S)$. Note that the initial definition by Gardiner and Masur used the subset of essential simple curves on $S$ instead of $\MF(S)$, but both yield isomorphic compactifications \cite[Lemma 19]{walsh2019asympgeom}. There is also a non-projective version where one chooses a particular scaling by fixing a basepoint in $\Teich(S)$.

\subsection{The horoboundary and the visual boundary} \label{subsec:horo}

We gave one definition of the horofunction compactification of a proper metric space $(W,d)$ in the introduction, but here is an equivalent one which is slightly more concrete. Fix a basepoint $b\in W$, then for every $y \in W$, associate the function $h_y:W \to \RR$ defined by $h_y(x) = d(x,y)-d(y,b)$. All these functions are $1$-Lipschitz and vanish at the point $b$. The map $y \mapsto h_y$ is an embedding (since $h_y$ has a unique minimum at $y$) and has compact closure with respect to the compact-open topology. This closure is the \emph{horofunction compactification} of $(W,d)$. Changing the basepoint yields isomorphic compactifications.

In the case where $(W,d)$ is the Teichm\"uller space $\Teich(S)$ equipped with the Teichm\"uller metric, Liu and Su \cite{LiuSu} proved that the horofunction compactification, which we will denote by $\overline{\Teich}^h(S)$, is isomorphic to the Gardiner--Masur compactification. We will often confound these two compactifications.

Given a basepoint $b\in \Teich(S)$ and any unit quadratic differential $q \in Q^1(b)$, the Teichm\"uller geodesic $R_q(t)$ converges to some limit $B(q)$ in the \emph{horoboundary} $\partial_h\Teich(S):=\overline{\Teich}^h(S) \setminus \Teich(S)$, as follows from a more general result of Rieffel \cite{Rieffel}. Conversely, every Busemann point (limit of an almost-geodesic ray) in $\partial_h\Teich(S)$ is equal to $B(q)$ for a unique $q \in Q^1(b)$ \cite[Theorem 6]{walsh2019asympgeom}. We will denote this set of Busemann points by $\calB$.

One can also define a visual compactification of Teichm\"uller space by adding a point at the end of each geodesic ray based at $b$ (see e.g. \cite{Kerckhoff} or \cite[Section 2.2]{azemar2021qualitative}). Azemar \cite[Theorem 1.1]{azemar2021qualitative} proved that there is a continuous map \[\Pi_b : \overline{\Teich}^h(S)  \to \Teich(S) \cup Q^1(b)\] from the horofunction compactification to this visual compactification such that $\Pi_b \circ B = \id$ on $Q^1(b)$, where $B$ is the Busemann map defined above. This means that any sequence in $\Teich(S)$ that converges to a point in the horoboun\-dary $\partial_h\Teich(S)$ has a well-defined limiting direction from the point of view of the basepoint $b$.

\section{Horofunctions outside the closure of Busemann points}  \label{sec:points_outside}

In this section, we mainly recall several results from \cite{azemar2021qualitative} that are needed for our proof.

The first result we need provides a necessary criterion for a point in the Gardiner--Masur boundary to be in the closure of Busemann points.

\begin{prop}[{\cite[Proposition 6.2]{azemar2021qualitative}}]\label{pr:closurecharacterization}
Let $S$ be a surface of finite type, let $X \in \Teich(S)$, let $q\in Q^1(X)$, let $\xi\in \overline{\mathcal{B}}\cap \Pi_X^{-1}(q)$, and let $V_j$ be the indecomposable components of the vertical foliation $V(q)$. Then the square of any representative of $\xi$ in the Gardiner--Masur compactification is a homogeneous polynomial of degree 2 in the variables $x_j=\frac{i(V_j,\cdot)}{i(V_j,H(q))}$, where $i$ denotes the intersection number.
\end{prop}

Recall that an \emph{indecomposable component} of a measured foliation is either a cylinder component or one of the ergodic parts of the transverse measure on a minimal component.

In particular, each coordinate of any $\xi\in \overline{\mathcal{B}}\cap \Pi_X^{-1}(q)$ (or more precisely, of  $\xi^2$ as a function in the deprojectivized Gardiner--Masur boundary) is smooth with respect to the variables $x_j$. Azemar uses this result to show that Busemann points are not dense within the horoboundary by finding points in some $\Pi_X^{-1}(q)$ which are not smooth in the $x_j$'s. These strange boundary points, which were first used in \cite{fortier2019Divergent} and which we will need again here, are obtained as limits of certain sequences of Dehn multitwists applied to any point in Teichm\"uller space.

\begin{lem}[Extension of {\cite[Corollary 3.4]{fortier2019Divergent}}]\label{lem:maxsresult}
		Let $\gamma_1 \cup\cdots \cup\gamma_k$ be a multicurve in a surface $S$ of finite type, let $\tau_j$ be either the left or right Dehn twist about $\gamma_j$, let $w=(w_1,\ldots,w_k)$ be a vector of positive weights, and for every $n\geq 1$, let $\phi_n=\tau_1^{\lfloor n w_1 \rfloor}\circ \dots \circ \tau_k^{\lfloor n w_k \rfloor}$. Then for every $X \in \Teich(S)$, the sequence $(\phi_n(X))_{n=1}^\infty$ converges to the projective class of the map
		\[
			F \mapsto \xi_{w,X}(F):=\EL^{1/2}\left(\sum_{j=1}^k w_j i(\gamma_j,F)\gamma_j,X\right)
		\]
		in the Gardiner--Masur compactification.
\end{lem}

The statement given in \cite[Corollary 3.4]{fortier2019Divergent} is less general but the same proof yields the above version with weights, as observed in \cite{azemar2021qualitative}. 

We will want to show that the conclusion of \Cref{pr:closurecharacterization} fails for certain limit points resulting from \Cref{lem:maxsresult}. In order to do so, we will need to know something about the quadratic differentials whose fiber contains such limit points.

\begin{lem}\label{le:pointinfiber}
With the same notation as in \Cref{lem:maxsresult}, the limit point
$[\xi_{w,X}]$ belongs to \(\Pi_X^{-1}(\psi)\) for some quadratic differential \(\psi\in Q^1(X)\) whose vertical foliation is a positive linear combination of the curves $\gamma_j$.
\end{lem}
\begin{proof}
By a result of Strebel \cite{strebel1966modular}, there is some $M>0$ and a unique quadratic differential \(\psi\in Q^1(X)\) whose vertical foliation decomposes $X$ into \(k\) cylinders $C_j$, each with with core curve \(\gamma_j\) and modulus (the ratio of width to circumference) equal to \(M w_j\).

A standard application of the pigeonhole principle implies that for any \(\varepsilon>0\) and any $N\in \NN$, there is an \(n \geq N \) such that \(\{ n w_j \}<\varepsilon\) for all \(j \in \{ 1, \ldots, k\}\), where $\{x\}:= x-\lfloor x \rfloor$ denotes the fractional part. It follows that there exists a sequence \( (n_{m})_{m=1}^\infty \subseteq \NN\) diverging to infinity such that \(  \{n_m w_j \} \to 0\) as $m\to\infty$, for every $j$.

Let
\[
h_{t}=\begin{pmatrix}
1 & 0 \\
t & 1
\end{pmatrix}
\]
be the vertical shear by $t$ with respect to the natural coordinates for \(\psi\). Observe that shearing the cylinder $C_j$ by $1/(M w_j)$ performs the left Dehn twist $\tau_j$ about $\alpha_j$. Thus, shearing $X$ by $n_m/M$ is the same as twisting $\lfloor n_m w_j \rfloor$ times and then shearing by 
 \( \{n_m w_j\} /(M w_j)\) in each cylinder $C_j$. The piecewise linear map obtained by performing this last shear in each cylinder has quasiconformal constant tending to $1$ as $m \to \infty$ since the shears tend to zero. Hence, the Teichm\"uller distance between $h_{n_m/M}(X)$ and $\phi_{n_m}(X)$ tends to zero as $m \to \infty$, where $\phi_n=\tau_1^{\lfloor n w_1 \rfloor}\circ \dots \circ \tau_k^{\lfloor n w_k \rfloor}$. Therefore, the accumulation points of the sequence \(\phi_{n_m}(X) \) in the horoboundary are contained in the accumulation points of \(h_t(X)\). Note that the path \(h_t(X)\) 
 is a horocycle in the Teichm\"uller disk $D$ through $\psi$ which converges to the same limit point as the ray $R_\psi$ in the visual compactification of $D$ (which is isometric to the hyperbolic plane). In particular, if we write $h_t(X)=R_{\psi_t}(s_t)$ for some $\psi_t \in Q^1(X)$ and $s_t \geq 0$, then $\psi_t \to \psi$ and $s_t \to \infty$ as $t \to \infty$, so that $h_t(X)\to \psi$  in the visual compactification of $\Teich(S)$ as $t \to \infty$. It follows that the accumulation points of the path $h_t(X)$ in the horoboundary $\partial_h\Teich(S)$ are all contained in \(\Pi_X^{-1}(\psi)\), hence so are those of the sequence $(\phi_{n_m}(X))_{m=1}^\infty$.

 By \Cref{lem:maxsresult}, the sequence $(\phi_{n_m}(X))_{m=1}^\infty$ converges to $\xi_{w,X}$ in the Gardiner--Masur compactification, and the above argument shows that the limit is contained in $\Pi_X^{-1}(\psi)$.
\end{proof}

In view of the above results, in order to find boundary points outside the closure of Busemann points, it suffices to find a multicurve $\bigcup \gamma_j$ for which 
\[
\EL\left(\sum_{j=1}^k x_j \gamma_j,X\right)
\]
is not a homogeneous polynomial of degree 2 in the $x_j$'s. The following example with two curves on a five-times punctured sphere was given in \cite[Lemma 7.2]{azemar2021qualitative}.

\begin{lem}[{\cite[Lemma 7.2]{azemar2021qualitative}}]\label{le:baseExample}
Let $X$ be a five-times punctured sphere which admits an anti-conformal involution $J$ fixing the five punctures and let $\alpha$ and $\beta$ be disjoint, non-homotopic, essential simple closed curves on $X$ that are both invariant under $J$. Then the function
\[t\to \EL\left(\alpha+t\beta,X\right)\]
defined for $t\geq 0$ is not $C^2$ at \(t=0\).
\end{lem}

This example can be exported to Teichm\"uller spaces of more complicated surfaces via covering constructions, as done in \cite[Theorem 1.8]{azemar2021qualitative}, to yield horofunctions outside the closure of Busemann points. We recall the reasoning here.

In what follows, we write $\overline{S_{g,p}}$ for the genus $g$ surface obtained by filling in the $p$ punctures of $S_{g,p}$, which then become marked points.

\begin{defn}
A branched cover $\overline{S_{g,p}}\to \overline{S_{h,q}}$ is \emph{admissible} if it sends marked points to marked points and is branched (i.e., not locally injective) at all unmarked preimages of marked points.
\end{defn}

We will need the fact that extremal length behaves well under admissible branched covers, as proved in \cite[p.1899--1900]{fortier2019Divergent} (see also \cite[Lemma 4.1]{bolza} and \cite[Lemma 7.3]{azemar2021qualitative}).

\begin{lem} \label{lem:Elcovers}
Let $\pi:\overline{S_{g,p}}\to \overline{S_{h,q}}$ be an admissible branched cover of degree $d$ and let $\iota_\pi:\Teich(S_{h,q})\hookrightarrow \Teich(S_{g,p})$ be the induced isometric embedding obtained by pulling back complex structures. Then for any measured foliation $F$ on $S_{h,q}$ and any $X\in \Teich(S_{h,q})$, we have
\[
\EL(\pi^{*}(F),\iota_\pi(X))=d \cdot \EL(F,X).
\]
\end{lem}

That an admissible branched cover induces an isometric embedding between Teichm\"uller spaces is shown for example in \cite[Theorem 5]{toy}. 

An immediate corollary of \Cref{lem:Elcovers} is that the example from \Cref{le:baseExample} can be lifted via admissible branched covers. The existence of admissible branched covers $\overline{S_{g,p}}\to \overline{S_{0,5}}$ for every non-negative integers $g$ and $p$ such that $3g+p>4$ was shown in \cite[Lemma 7.1]{GekhtmanMarkovic}, and this suffices to show that Busemann points are not dense in the horoboundary of $\Teich(S_{g,p})$.

However, to show that the set $\calB$ of Busemann points is nowhere dense in the horoboundary $H=\partial_h \Teich(S_{g,p})$, we need to prove that the complement of $\overline{\calB}$ in $H$ is dense in $H$. Hence, we need to find enough points in $H \setminus \overline{\calB}$ to accumulate everywhere. Our strategy for this is to find a $1$-parameter family of simplices $\Delta_s \subset H \setminus \overline{\calB}$ which accumulates everywhere on a simplex $\Delta \subset \calB$ of Busemann points corresponding to Jenkins--Strebel differentials as $s \to \infty$. We then show that the mapping class group orbit of this simplex $\Delta$ accumulates onto all the Busemann points provided that it has maximal dimension $3g+p-4$, or equivalently,  corresponds to Jenkins--Strebel differentials whose core curves form a maximal multicurve.

For this strategy to work, we need to choose admissible branched covers $\pi:\overline{S_{g,p}}\to \overline{S_{0,5}}$ such that the pullback $\pi^*(\alpha + \beta)$ is a weighted maximal multicurve, which we do in the next section.

\section{Simplices outside the closure} \label{sec:simplices}

As stated above, our goal is to find,  whenever $3g+p>4$, an admissible branched cover $\overline{S_{g,p}}\to \overline{S_{0,5}}$ such that the pullback of the pants decomposition $\alpha+\beta$ of $S_{0,5}$ from \Cref{le:baseExample} forms a weighted pair of pants decomposition on $S_{g,p}$.

\begin{lem}\label{lem:preciselift}
Suppose that $3g+p>4$ and let $P_5$ be a pants decomposition of $S_{0,5}$. Then there exists an admissible branched cover $\pi : \overline{S_{g,p}}\to \overline{S_{0,5}}$  such that $\pi^{*}(P_5)$ is measure-equivalent to a pants decomposition of $S_{g,p}$ with positive weights.
\end{lem}
\begin{proof}
The idea of the proof, similarly as in \cite[Lemma 7.1]{GekhtmanMarkovic}, is to construct $\pi$ inductively as a composition of admissible branched covers of degree $2$. If $S$ is a punctured surface and $A \subset S$ is a collection of disjoint simple proper arcs, then we can construct a covering map $f:R \to S$ of degree $2$ by taking two copies of $S$ cut along $A$ and gluing each side of each slit in one copy to the other side of the same slit in the other copy. The resulting surface $R$ is connected if and only if the union of the arcs does not separate $S$. In that case, the genera of $R$ and $S$ satisfy $g(R)=g(S)+|A|-1$. The map $f$ extends to an admissible branched cover $F$ between the compactifications. We can further unmark the preimages in $\overline{R}$ of marked points in $\overline{S}$ that are adjacent to only one arc in $A$ since $F$ is not locally injective at such points. 

Let $P$ be a set of curves that form a pants decomposition of $S$. We say that $A$ is \emph{compatible} with $P$ if the extremities of the arcs in $A$ are all distinct and each arc of $A$ intersects at most one curve in $P$ and does so at most once (transversely). In that case, $F^*(P)$ is measure-equivalent to a pants decomposition of the surface $T$ obtained by removing the marked points from $\overline{R}$, provided that we have unmarked the preimages of both extremities of each arc that intersects $P$ and the preimage of at least one extremity of each arc that is disjoint from $P$. Indeed, each pair of pants in $S\setminus P$ disjoint from $A$ lifts to two disjoint pairs of pants in $T$. If a pair of pants contains an arc $\alpha$ in $A$, then its preimage in $\overline{R}$ a cylinder. If we unmark the preimages of both extremities of $\alpha$, then the cylinder is contained in $T$, hence its two boundary curves are homotopic and can be merged together into a single curve with twice the weight.  If we unmark the preimage of only one extremity of $\alpha$, then the cylinder becomes a $1$-cusped pair of pants in $T$.  Finally, if a pair of pants $\Pi$ intersects an arc in $A$ but does not contain it, then the hypothesis that $A$ is compatible with $P$ implies that $\Pi \cap A$ is an arc between a puncture and a pants curve. In that case, the double branched cover of $\Pi$ becomes a pair of pants after unmarking the branch point, similarly as in the previous case. Therefore, after merging homotopic components together, all complementary regions become pairs of pants. All three cases are illustrated in \Cref{fig:lifts}.

\begin{figure}
\centering
\includegraphics[width=1\textwidth]{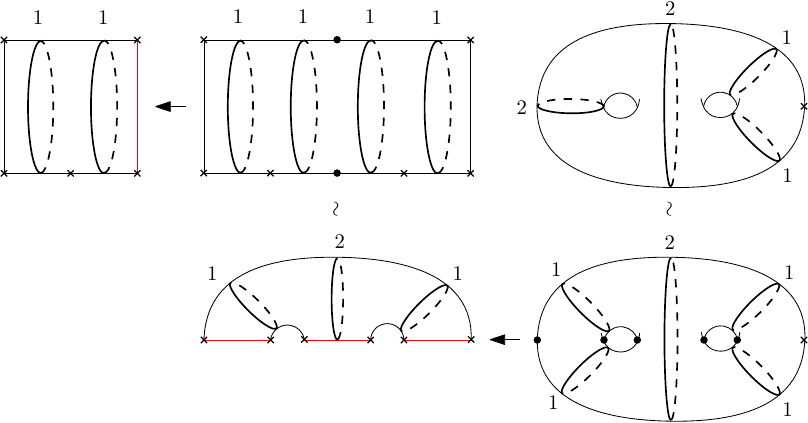}
\caption{Some double branched covers arising in the proof of \Cref{lem:preciselift}. Marked points are indicated by crosses and unmarked branched points by dots. The arc systems used to define the double branched covers are shown in red. The numbers indicate weights on curves.}
\label{fig:lifts}
\end{figure}

To prove the lemma, we first show that for each $p\geq 5$, there exists an admissible branched cover $\pi : \overline{S_{0,p}} \to \overline{S_{0,5}}$ such that $\pi^*(P_5)$ is measure-equivalent to a pants decomposition $P$ which divides $S_{0,p}$ into two $2$-cusped pants connected by a sequence of $1$-cusped pants. This statement, call it $B(p)$, is clearly true for $p=5$. Now suppose $B(p)$ is true for some $p\geq 5$ and let $\pi_p$ be the associated branched cover and let $P_p$ be the associated pants decomposition. Now take a proper simple arc $\alpha$ in $S_{0,p}$ disjoint from $P_p$ (contained in one of the two $2$-cusped pants) and construct the associated double branched cover $F$ as above. Since the arc system $A=\{\alpha\}$ is compatible with $P_p$, we have that $F^*(P_p)$ is measure-equivalent to a pants  decomposition $Q$ on the double cover minus its marked points, which is either $S_{0,2p-4}$ or $S_{0,2p-3}$ depending on how many points we unmark. Lastly, the arguments of the previous paragraph imply that $Q$ still divides this surface into two $2$-cusped pants connected by a sequence of $1$-cusped pants. The composition $ \pi:= \pi_p \circ F$ then shows that $B(2p-4)$ and $B(2p-3)$ are true. Thus if $B(p)$ is true for $p = 4 + 2^j$ to $p=3+ 2^{j+1}$, then it is also true for $p = 2(4+2^j)-4 = 4+ 2^{j+1}$ to $p=2(3+2^{j+1})-3 = 3+2^{j+2}$, which covers all the integers $p\geq 5$ when we start with $j=0$.

For any $q\geq 5$ and a pants decomposition $Q$ on $S_{0,q}$ obtained as above, let $\overline{A}$ be an arc system on $S_{0,q}$ consisting of the two proper arcs joining the two punctures in each of the two $2$-cusped pants, together with $\lfloor(q-4)/2\rfloor$ arcs joining punctures of consecutive $1$-cusped pants. This arc system is compatible with $Q$ and has cardinality $\lfloor q/2 \rfloor$. If $A$ is any subset of $\overline{A}$ of cardinality $k \leq \lfloor q/2 \rfloor$ which contains at least one of the first two arcs, then the double cover construction applied to $A$ yields a surface of genus $k-1$ with $2(q-2k)$ marked points if we unmark all the branch points, or $2(q-2k)+1$ marked points if we keep one branch point marked. Once again, the composition of this branched cover with the previous branched cover $\pi_q: \overline{S_{0,q}} \to \overline{S_{0,5}}$ shows that the statement of the lemma is satisfied for $(g,p) = (k-1,2(q-2k))$ and $(g,p) = (k-1,2(q-2k)+1)$.

We can now finish the proof. Let $g\geq 0$ and $p\geq 0$ be such that $3g+p \geq 5$. If $g=0$, then we are done since $B(p)$ it true, so we may assume that $g\geq 1$. If $p$ is even, write $p=2n$ and let $q := n+2(g+1)$, which is at least $5$ since $(g,p)\neq (1,0)$, and let $k := g+1 \leq \lfloor q/2 \rfloor$. By the previous paragraph, the statement of the lemma is satisfied for this pair $(g,p)$. If $p$ is odd, write $p=2n+1$ and let $q:=n+2(g+1)$ and $k:=g+1$. Once again $q\geq 5$ since $(g,p)\neq (1,1)$ and we still have $k \leq \lfloor q/2 \rfloor$, so the above reasoning shows that the pair $(g,p)$ can be obtained.
\end{proof}

Using these particular branched covers and the results of the previous section, we obtain simplices of large dimension outside the closure of Busemann points.

\begin{cor}\label{co:outsideclosure}
 Suppose that $3g+p>4$. Then there is a a weighted pair of pants decomposition $\gamma=\sum_{j=1}^k \lambda_i \gamma_i$ on $S_{g,p}$ and a non-empty totally geodesic subset $G \subset \Teich(S_{g,p})$ such that for any $X \in G$, and any vector $w = (w_1,\ldots,w_k)$ of positive weights, the limit point $[\xi_{w,X}]$ from \Cref{lem:maxsresult} is not in the closure of the set of Busemann points. Furthermore, for any $X\in G$, the Teichm\"uller ray in the direction of the quadratic differential $q \in Q(X)$ with $V(q)= \gamma$ is contained in $G$.
\end{cor}
\begin{proof}
Let $J:S_{0,5} \to S_{0,5}$ be an orientation-reversing involution  that fixes the five punctures and let $P_5 = \alpha + \beta$ be a pants decomposition whose components are both invariant under $J$. The subset $D \subset \Teich(S_{0,5})$ of conformal structures for which $J$ is homotopic to an anti-conformal homeomorphism is totally geodesic (it isometric to the space of conformal pentagons with labelled vertices). 

Let $\pi : \overline{S_{g,p}}\to \overline{S_{0,5}}$ be an admissible branched cover provided by \Cref{lem:preciselift}, let $d$ be its degree, and let $\iota_\pi:\Teich(S_{0,5})\hookrightarrow \Teich(S_{g,p})$ be the induced isometric embedding. By construction, $\pi^*(P_5)$ is measure-equivalent to a weighted pants decomposition $\gamma = \sum_{j=1}^k \lambda_j \gamma_j$ of $S_{g,p}$. After relabelling the curves if necessary, there is some $m$ such that $\pi^*(\alpha)=\sum_{j\le m} \lambda_j\gamma_j$ and $\pi^*(\beta) =\sum_{j> m} \lambda_j\gamma_j$.

Since $\iota_\pi$ is an isometric embedding, the set $G:=\iota_\pi(D)$ is totally geodesic in $\Teich(S_{g,p})$. Furthermore, for any $Y \in D$, the unique quadratic differential $\phi$ on $Y$ with vertical foliation equal to $\alpha + \beta$ is invariant under $J$, hence so is the Teichm\"uller ray in the direction of $\phi$, so the latter is contained in $D$. The image of that geodesic ray by $\iota_\pi$ is the Teichm\"uller ray in the direction of $q:=\pi^*\phi$, which has vertical foliation $\pi^*(\alpha+\beta) = \gamma$. This Teichm\"uller ray is thus contained in $G =\iota_\pi(D)$.

Now fix a vector of positive weights \(w=(w_1,\ldots,w_k)\) and an \(X=\iota_\pi(Y) \in G\). By \Cref{le:pointinfiber}, there is some \(\psi \in Q^1(X)\) with \(V(\psi)=\sum_{j=1}^k d_j \gamma_j\) for some positive $d_j$ such that \([\xi_{w,X}] \in \Pi_X^{-1}(\psi)\).

Suppose that \([\xi_{w,X}]\) is in the closure of Busemann points. Since \([\xi_{w,X}]\) belongs to \(\Pi_X^{-1}(\psi)\), the square \(\xi_{w,X}^2\) is a polynomial of degree $2$ in the variables \[x_j(\cdot)=\frac{i(\gamma_j,\cdot)}{i(\gamma_j,H(\psi))}\] according to \Cref{pr:closurecharacterization}. Writing \(\xi_{w,X}^2\) in term of these variables, we find that
\[\xi_{w,X}(F)^2=
\EL\left(\sum_j w_j x_j(F) i(\gamma_j,H(\psi))\gamma_j,Y \right)
\]
for every measured foliation $F$. However, we will see that the last expression is not smooth with respect to the variables \(x_j\). 

Recall that for any pants decomposition $\gamma = \sum_{j=1}^k  \gamma_j$ of a surface $S$ and any non-negative numbers $c_1,\ldots,c_k$, there is a measured foliation $F$ such that $i(\gamma_j,F)=c_i$ for every $c$. Indeed, one can construct $F$ in each pair of pants using \cite[Proposition 6.7]{FLP} and then glue the pieces together to get a global measured foliation. It follows that for each $t\geq 0$, there is a measured foliation \(F_t\) on $S_{g,p}$ such that \(x_j(F_t)=\frac{\lambda_j}{w_j i(\gamma_j,H(\tilde{q}))}\) for \(j\le m\) and \(x_j(F_t)=t\frac{\lambda_j}{w_j i(\gamma_j,H(\tilde{q}))}\) for $j>m$.

Observe that the $x_j$ depend affinely on $t$ (hence smoothly), yet
\begin{align*}
\xi_{w,X}(F_t)^2 &=
\EL\left(\sum_j w_j x_j(F_t) i(\gamma_j,H(\psi))\gamma_j,X\right) \\ 
&= \EL\left(\sum_{j \le m} \lambda_j \gamma_j + t\sum_{j > m} \lambda_j \gamma_j , X \right) \\
&= \EL\left(\pi^*(\alpha + t \beta) , \iota_\pi(Y) \right) \\
& = \EL(\alpha + t \beta , Y)
\end{align*}
is not smooth at $t=0$ by \Cref{le:baseExample}. It follows that \(\xi_{w,X}^2\) is not a smooth function of the $x_j$'s, hence is not a polynomial, which is a contradiction. We conclude that $[\xi_{w,X}]$ is not in the closure of the set of Busemann points.
\end{proof}

Observe that scaling $w$ does not change the projective class of the limit point $\xi_{w,X}$, hence for each fixed $X\in G$, the previous result yields an open simplex 
\[
\Delta_X := \left\{ [\xi_{w,X}] : w_i > 0 \text{ for all $i$ and } \sum_i w_i = 1 \right\}
\]
contained in $\partial_h \Teich(S_{g,p}) \setminus \overline{\calB}$. We now want to show that as $X$ tends to infinity along the Teichm\"uller ray $R_q$, the simplex $\Delta_X$ accumulates onto the whole simplex of Busemann points associated to quadratic differentials whose vertical foliations are reweighings of $\gamma$.

To state this more precisely, we need some more notation. For a unit quadratic differential $\phi$, Walsh showed in \cite[Corollary 1]{walsh2019asympgeom} that the Busemann point $B(\phi)$ corresponds to the projective class of 
\[
\calE_\phi(F) = \left(\sum_j \frac{i(G_j,F)^2}{i(G_j,H(\phi))}\right)^{1/2}
\]
in the Gardiner--Masur compactification, where the $G_j$ are the indecomposable components of $V(\phi)$. Our statement is then as follows.

\begin{prop}\label{pr:simplexoutsideclosure}
Suppose that $3g+p>4$. Let $G \subset \Teich(S_{g,p})$ and \(\gamma =\sum_{j=1}^k \lambda_j \gamma_j\) be as in \Cref{co:outsideclosure}, let $X\in G$, let $q\in Q^1(X)$ be such that $V(q)$ is proportional to $\gamma$, and let \(\phi \in Q^1(X)\) be such that $V(\phi)$ is a positive linear combination of the $\gamma_i$. Then there is a vector \(w=(w_1,\ldots,w_k)\) of positive weights such that
\[\lim_{s\to \infty} e^{s}\xi_{w,\trayt{q}{s}}(F)=\calE_\phi(F)\]
for every $F \in \MF(S_{g,p})$. In particular, $[\xi_{w,\trayt{q}{s}}]$ converges to the Busemann point $B(\phi)$ as $s\to \infty$, so that $B(\phi)$ is not in the interior of the closure of the set of Busemann points. 
\end{prop}

The proof requires finding the limit of the extremal length of a foliation when moving along a geodesic ray, which in turn requires the following results from Walsh's paper \cite{walsh2019asympgeom}.

\begin{lem}[{\cite[Lemma 4]{walsh2019asympgeom}}]\label{le:walshlowerbound}
Let \(q\) be a unit area quadratic differential with vertical indecomposable components \(G_1,\ldots,G_k\). Suppose that \(F=\sum f_j G_j\) for some \(f_j\ge 0\). Then,
\[\liminf_{t\to \infty} e^{2t}\EL(F,\trayt{q}{t})\ge\sum_j f_j^2 i(G_j,H(q)).\]
\end{lem}

\begin{lem}[{\cite[Lemma 3]{walsh2019asympgeom}}]\label{th:walshlimit}
Let \(q\) be a unit area quadratic differential with vertical indecomposable components \(G_1,\ldots,G_k\). Then for any measured foliation $F$ and any $t\geq 0$, we have
\[e^{-2t}\EL(F,\trayt{q}{t})\ge\sum_j\frac{i(G_j,F)^2}{i(G_j,H(q))}.\]
\end{lem}

\begin{lem}[Titu's lemma]\label{le:tituslemma}
Let $n\ge 1$, and let \(a,b \in [0,\infty)^n \setminus\{0\}\) be such that there is no coordinate \(j\) for which \(a_j\) and \(b_j\) are both zero. Then the function from \([0,\infty)^n \setminus\{0\}\) to \(\RR\) defined by
\[x\to \frac{\left(\sum_{j=1}^n a_j x_j\right)^2}{\sum_{n=1}^n b_j x_j^2}\]
attains its maximum of \(\sum_{j=1}^n a_j^2/b_j\) only when there is a constant $C>0$ such that \(x_j=Ca_j/b_j\) for all $j$.
\end{lem}

We combine these results to obtain the rescaled limit of the extremal length of foliations that get pinched along a Teichm\"uller ray. This should be compared with \cite[Theorem 1]{walsh2019asympgeom}, which gives the rescaled limit of the extremal length for foliations that get stretched along the Teichm\"uller ray.

\begin{lem}\label{le:extremallengthlimit}
Let \(q\) be a unit area quadratic differential with vertical indecomposable components \(G_1,\ldots,G_k\) and let \(F=\sum f_j G_j\) for some \(f_j\ge 0\). Then
\[\lim_{t\to \infty} e^{2t}\EL(F,\trayt{q}{t})=\sum_j f_j^2 i(G_j,H(q)).\]
\end{lem}
\begin{proof}
By \Cref{le:walshlowerbound} we only need to prove the upper bound. 
For each $t \in \RR$, let $F_t$ be the horizontal foliation of the quadratic differential on $\trayt{q}{t}$ with vertical foliation $F$. Then
\[\EL(F,\trayt{q}{t})\EL(F_t,\trayt{q}{t})=i(F,F_t)^2\]
by the case of equality in Minsky's inequality (see \cite{gardmasur1991extrgeom}).
Thus,
\begin{align*}
\limsup_{t\to \infty}e^{2t}\EL(F,\trayt{q}{t}) &= \limsup_{t\to \infty}\frac{i(F,F_t)^2}{e^{-2t}\EL(F_t,\trayt{q}{t})} \\
&\le \limsup_{t\to \infty}\sup_{G\in \MF}\frac{i(F,G)^2}{e^{-2t}\EL(G,\trayt{q}{t})} \\
&\le\sup_{G\in \MF}\frac{i(F,G)^2}{\sum_j\frac{i(G_j,G)^2}{i(G_j,H(q))}},
\end{align*}
where the last inequality is from \Cref{th:walshlimit}. By applying \Cref{le:tituslemma} with $a_j = f_j$, $b_j=1/i(G_j,H(q))$, and $x_j = i(G_j,G)$, we obtain
\[
\frac{ i(F,G)^2}{ \sum_j \frac{i(G_j,G)^2}{i(G_j,H(q))}} = \frac{\left(\sum_j f_j i(G_j,G)\right)^2}{ \sum_j \frac{i(G_j,G)^2}{i(G_j,H(q))}} \leq \sum_j f_j^2 i(G_j,H(q))
\]
for every $G$, which yields the desired inequality.
\end{proof}

We can now prove our result concerning the accumulation points of the simplex $\Delta_{\trayt{q}{s}}$ as $s\to \infty$.

\begin{proof}[Proof of \Cref{pr:simplexoutsideclosure}]
Recall that \(V(q)=c \sum_j \lambda_j \gamma_j\) for some $c>0$, which has indecom\-po\-sable components $c\lambda_j \gamma_j$. Thus, for any measured foliation \(F\) and any weight vector $w$, we have
\begin{align*}
\lim_{s\to \infty} e^{2s}\xi_{w,\trayt{q}{s}}(F)^2&=\lim_{s\to \infty} e^{2s}\EL\left(\sum_j w_j i(\gamma_j,F)\gamma_j, \trayt{q}{s}\right)\\
&=\sum_j\left(\frac{w_j i(\gamma_j,F)}{c\lambda_j}\right)^2 i(c\lambda_j\gamma_j,H(q)) \\
&=\sum_j \frac{w_j^2 i(\gamma_j,F)^2}{c\lambda_j} i(\gamma_j,H(q))
\end{align*}
by \Cref{le:extremallengthlimit}.

On the other hand, for any quadratic differential \(\phi\) of unit area on $X$ with \(V(\phi)=\sum_j d_j \gamma_j\), we have
\[\calE_\phi(F)^2=\sum_j \frac{i(d_j\gamma_j,F)^2}{i(d_j\gamma_j,H(\phi))}=\sum_j \frac{d_j i(\gamma_j,F)^2}{i(\gamma_j,H(\phi))},\]
so setting \[w_j=\left(\frac{ c\lambda_j d_j}{i(\gamma_j,H(q))i(\gamma_j,H(\phi))}\right)^{1/2}\]
yields that
\[
\lim_{s\to \infty} e^{s}\xi_{w,\trayt{q}{s}}(F) = \calE_\phi(F)
\]
as required. This means that $[\xi_{w,\trayt{q}{s}}] \to [\calE_\phi]=B(\phi)$ as $s\to \infty$. In particular, any open neighborhood of $B(\phi)$ contains points of the form $[\xi_{w,\trayt{q}{s}}]$, which are not in $\overline{\calB}$ according to \Cref{co:outsideclosure}. In other words, $B(\phi)$ does not belong to $\overline{\calB}^\circ$, where the interior is taken relative  to the horoboundary (as opposed to the whole compactification).
\end{proof}

Let $Q_X^\gamma \subset Q^1(X)$ be the set of quadratic differentials of unit area on $X$ whose vertical foliation is a positive linear combination of the $\gamma_i$. Now that we have found one simplex $B(Q_X^\gamma)$ disjoint from $\overline{\calB}^\circ$, we can obtain many others using the action of the mapping class group, and then we can take the closure of this orbit. In the next section, we show that this orbit closure contains $\calB$, hence $\overline{\calB}$, which implies that 
 $\overline{\calB}^\circ$ is empty. In other words, Busemann points are nowhere dense within the boundary.

\section{Orbits of top-dimensional simplices}  \label{sec:orbits}

The final step in the proof is to show that the mapping class group orbit of the simplex $B(Q_X^\gamma)$ is dense among Busemann points. For this it, is useful to recall a result of Walsh which characterizes when a sequence of Busemann points converges to a Busemann point. In order to state the result, we first introduce some terminology.

\begin{defn} \label{def:strong_conv}
Let $S$ be a surface of finite type and let $F, F_n \in \MF(S)$. We say that $F_n$ \emph{converges strongly} to $F$ as $n \to \infty$ if the following conditions hold:
\begin{enumerate}
\item $F_n$ converges to $F$ in the usual (weak) sense as $n \to \infty$;
\item  For any sequence $F_n^{j_n}$ of indecomposable components  of $F_n$, any weak limit of a subsequence of this sequence has to be indecomposable. 
\end{enumerate}
\end{defn}
Recall that the zero foliation is considered to be indecomposable, so some indecomposable components of $F_n$ are allowed to disappear in the limit. Informally, the second condition means that no indecomposable component of $F_n$ can ``split'' in the limit. However, distinct indecomposable components are allowed to merge in the limit.

The convergence of the Busemann points associated to geodesic rays from a common basepoint can be characterized as follows.

\begin{thm}[{\cite[Theorem 10]{walsh2019asympgeom}}] \label{walsh_Busemann}
Let $S$ be a surface of finite type, let $X \in \Teich(S)$, and let $q, q_n \in Q^1(X)$. Then $B(q_n) \to B(q)$ as $n \to \infty$  if and only if $V(q_n)$ converges strongly to $V(q)$ as $n \to \infty$.
\end{thm}

Note that the topology on $\calB$ can be pulled back to $Q^1(X)$ via the Busemann map $B$, and then to $\PMF(S)$ via the Hubbard--Masur map. In turn, this defines a topology on $\MF(X)$, which is an infinite cone over $\PMF(S)$.  It follows from the above theorem that the notion of strong convergence in $\MF(S)$ is compatible with this topology, which we thus call the \emph{strong topology}.

While weighted simple closed curves are dense in the weak topology on $\MF(S)$, they can only converge strongly to indecomposable foliations. In particular, weighted simple closed curves are not dense in the strong topology unless $S$ is a sphere with $4$-times punctured sphere or a once-punctured torus. However, we now show that weighted multicurves are still dense. In fact, the mapping class group orbit of any single cone of weighted maximal multicurves is dense.

\begin{prp}\label{pr:denseorbitsimplex}
Let $\alpha=\sum_j \alpha_j$ be a pants decomposition of a surface $S$ of finite type. Let $C(\alpha) \subset \MF(S)$ be the cone of positive linear combinations of the curves $\alpha_j$. Then the set $\bigcup_{f \in \MCG(S)} f(C(\alpha))$ is dense in  $\MF(S)$ with respect to the strong topology, where $\MCG(S)$ is the mapping class group of $S$.
\end{prp}
\begin{proof}
Let $\beta=\beta_1+\alpha_2+\cdots+\alpha_k$ be a pants decomposition adjacent to $\gamma$ in the pants graph (meaning that $i(\alpha_1,\beta_1)$ is equal to $1$ or $2$). Let $\tau$ be the Dehn twist around $\beta_1$. Then $\tau^n(\alpha_1)/n$ converges weakly to \(i(\alpha_1,\beta_1)\beta_1\) as $n\to \infty$, hence strongly since the limit is indecomposable. Furthermore, if \(j\ge 2\) then $\tau^n(\alpha_j)=\alpha_j$ for every $n$ since \(\beta_1\) and \(\alpha_j\) are disjoint. 

Let \(v=w_1 \beta_1+w_2\alpha_2+\ldots +w_k\alpha_k\) be any point in \(C(\beta)\). Then \[v_n = \tau^n\left(\frac{w_1}{n i(\alpha_1,\beta_1)} \alpha_1+\sum_{j=2}^k w_j \alpha_j\right)=\frac{w_1}{n i(\alpha_1,\beta_1)}\tau^n(\alpha_1)+\sum_{j=2}^k w_j \alpha_j,\] which converges strongly to \(v\) since with this labelling, the $j$-th indecomposable component of $v_n$ converges to the $j$-th indecomposable component of $v$. Thus, for any multicurve \(\beta\) adjacent to \(\alpha\) in the pants graph, we have 
\begin{equation} \label{eq:pants}
C(\beta) \subseteq \overline{\bigcup_{f \in \MCG(S)} f(C(\alpha))},
\end{equation}
where the closure is with respect to the strong topology.

It follows that 
\[\overline{\bigcup_{f \in \MCG(S)} f(C(\beta))} \subseteq \overline{\bigcup_{f \in \MCG(S)} f(C(\alpha))}.\]
because the right-hand side is mapping class group invariant and closed. Since the pants graph is connected \cite[Theorem 1.1]{HatcherThurston}, we can make chains of inclusions of this form and conclude that \Cref{eq:pants} hold for \emph{any} pants decomposition $\beta$ of $S$.

Now let $\gamma$ be a multicurve which is not maximal, and let $\tilde{\gamma}$ be a pants decomposition containing $\gamma$. By decreasing the weights of the extra curves to zero, we see that $C(\gamma)$ is contained in the closure of $C(\tilde{\gamma})$, hence also in $\overline{\bigcup_{f \in \MCG(S)} f(C(\alpha))}$.

Finally, let $\gamma \in \MF(S)$ be any measured foliation different from zero (which is obviously in the closure of any cone). Then $\gamma$ is a union of subfoliations $\gamma_j$ which are either weighted closed curves or minimal, all supported on disjoint subsurfaces $S_j$. In turn, each minimal subfoliation $\gamma_j$ is a finite sum of independent ergodic measures (which are the indecomposable components of $\gamma_j$). By \cite[Theorem C]{LenzhenMasur}, for any minimal foliation $F$ on a surface $R$ with ergodic components $\nu^1,\ldots, \nu^k$, there is a sequence of weighted multicurves $\delta_n = \sum_i \delta_n^i$ such that $\delta_n^i \to \nu^i$ as $n \to \infty$ for each $i$. In particular, $\delta_n$ converges strongly to $F$ as $n\to \infty$.  By applying this result to each minimal subfoliation $\gamma_j$ of $\gamma$ and adding the curve components of $\gamma$, we obtain a sequence of weighted multicurves $\epsilon_n$ on $S$ which converges strongly to $\gamma$ as $n\to \infty$.

We conclude that $\gamma$ is in the closure of the sequence $(\epsilon_n)_{n=1}^\infty$, which is itself in the closure of the orbit of $C(\alpha)$. This concludes the proof.
\end{proof}

\begin{rem}
We emphasize that the proposition is false for cones spanned by measured foliations with fewer indecomposable components, since these cannot accumulate onto foliations with a larger number of indecomposable components in the strong topology. By the same argument, it follows from \Cref{pr:denseorbitsimplex} that the number of indecomposable components of any measured foliation on $S_{g,p}$ is bounded above by the cardinality of a pants decomposition, which is $3g+p-3$.
\end{rem}

It remains to apply the above result to prove that the mapping class group orbit of the simplex $B(Q_X^\gamma)$ is dense among Busemann points. One slight issue is that most mapping classes do not preserve the basepoint $X$, so we cannot apply \Cref{walsh_Busemann} directly, but this is easy to overcome.

We first recall the definition of modular equivalence from Walsh's paper.

\begin{defn}
Two quadratic differentials $q$ and $q'$ on possibly different surfaces in $\Teich(S)$ are \emph{modularly equivalent} if there are indecomposable foliations $G_j$ and constants $C,a_j,a_j' >0$ such that $V(q)=\sum_j a_j G_j$, $V(q')=\sum_j a_j G_j$, and
\[
\frac{a_j}{i(G_j,H(q))}=C\frac{a_j'}{i(G_j,H(q'))}
\]
for every $j$.
\end{defn}

Walsh proved in \cite[Theorem 4]{walsh2019asympgeom} that if $X,Y \in \Teich(S)$, $\phi\in Q^1(X)$, and $\psi \in Q^1(Y)$, then $B(\phi)=B(\psi)$ if and only if $\phi$ and $\psi$ are modularly equivalent. Moreover, every modular equivalence class of quadratic differentials of unit area contains a unique representative at any point in Teichm\"uller space \cite[Theorem 5]{walsh2019asympgeom}. 
We use these results to change basepoints for our simplices.

\begin{lem}\label{co:imageMCG}
Let \(G\in \MF(S)\) be a measured foliation with indecomposable components \(G_j\). For \(X\in \Teich(S)\), denote by \(Q_X^G \subset Q^1(X)\) the set of all quadratic differentials of unit area on \(X\) whose vertical foliation is a positive linear combination of the $G_j$. Then $B(Q_X^G) = B(Q_Y^G)$ for every $X,Y \in \Teich(S)$. In particular, if \(f\in \MCG(S)\), then
\[
f\left(B\left(Q_X^G\right)\right)=B\left( f\left(Q_X^G\right)\right)=B\left(Q_{f(X)}^{f(G)}\right) = B\left(Q_X^{f(G)}\right).
\]
\end{lem}
\begin{proof}
Let $X,Y \in \Teich(S)$, let \(q\in Q_X^G\), and let $q' \in Q^1(Y)$ be modularly equivalent to $X$. By definition of modular equivalence, $V(q')$ is a positive linear combination of the $G_j$, so  $q'\in Q_Y^G$. Furthermore, $B(q) = B(q')$, so that $B\left(Q_X^G\right) \subseteq B\left(Q_Y^G\right)$. The reverse inclusion follows by switching $X$ and $Y$.

If $f\in \MCG(S)$, then $f\circ B = B \circ f$ since $f$ sends any geodesic ray and its limit point to a geodesic ray and its limit point since the mapping class group acts by isometries (which extend continuously to the horofunction compactification). Then the equality $f\left(Q_X^G\right)=Q_{f(X)}^{f(G)}$ simply follows from the action of the mapping class group on quadratic differentials, and by the previous paragraph $B\left(Q_{f(X)}^{f(G)}\right) = B\left(Q_X^{f(G)}\right)$.
\end{proof}

We now have all the ingredients to prove that Busemann points are nowhere dense in the horoboundary.
\begin{proof}[Proof of \Cref{th:maintheorem}]
Let \(\gamma\) the weighted pair of pants decomposition and $G \subset \Teich(S_{g,p})$ the totally geodesic subset from \Cref{co:outsideclosure}. Then let $X \in G$ and let \(\psi \in Q^1(X)\). We want to show that $B(\psi)$ is in the closure of the orbit of $B(Q_X^\gamma)$.

By \Cref{pr:denseorbitsimplex}, there exist sequences $f_n \in \MCG(S_{g,p})$ and $\gamma_n \in C(\gamma)$ (the set of reweighings of $\gamma$) such that \(F_n:=f_n(\gamma_n)\) converges strongly to $V(\psi)$. Let $\widetilde{\psi}_n \in Q(X)$ be such that $V\left(\widetilde{\psi}_n\right)=F_n$ and let $\psi_n = \widetilde{\psi}_n / \left\|\widetilde{\psi}_n \right\|$. Since the Hubbard--Masur map is a homeomorphism, $\widetilde{\psi}_n \to \psi$ as $n \to \infty$. In particular, $V(\psi_n) = F_n / \left\|\widetilde{\psi}_n\right\|$ converges strongly to $V(\psi)$ because $\left\|\widetilde{\psi}_n\right\| \to 1$. By Walsh's \Cref{walsh_Busemann}, $B(\psi_n) \to B(\psi)$ as $n \to \infty$.

Since \(\gamma_n \in C(\gamma)\), we have that $Q_X^{\gamma_n} = Q_X^\gamma$ for every $n$. Furthermore, by \Cref{co:imageMCG} we have
\[ f_n\left(B\left(Q_X^{\gamma}\right)\right)=f_n\left(B\left(Q_X^{\gamma_n}\right)\right)=B\left(Q_X^{f_n(\gamma_n)}\right)=B\left(Q_X^{F_n}\right).\]
As $\psi_n \in Q_X^{F_n}$, this means that for each $n$ there is a \(\zeta_n\in B\left(Q_X^{\gamma}\right)\) such that \(f_n(\zeta_n)=B(\psi_n)\). By \Cref{pr:simplexoutsideclosure}, for each $n$, \(\zeta_n\) is not in the interior of \(\Bclose\). Now \(f_n(\Bclose^\circ)= \Bclose^\circ\) because $f_n$ acts by homeomorphisms, so the same is true for \(f_n(\zeta_n)=B(\psi_n)\), and hence the same is true for the limit $B(\psi)$ of this sequence.

Since all the Busemann points are attainable through rays from any basepoint \cite[Theorem 6]{walsh2019asympgeom}, this shows that $\calB$ is disjoint from $\Bclose^\circ$, hence $\Bclose$ is disjoint from $\Bclose^\circ$, which means that $\Bclose^\circ$ is empty since it is a subset of $\Bclose$. In other words, Busemann points are nowhere dense in the horoboundary.
\end{proof}

\bibliographystyle{amsalpha}
\bibliography{BpointsNowhereDense}{}	

\newcommand{\etalchar}[1]{$^{#1}$}
\providecommand{\bysame}{\leavevmode\hbox to3em{\hrulefill}\thinspace}
\providecommand{\MR}{\relax\ifhmode\unskip\space\fi MR }
\providecommand{\MRhref}[2]{%
  \href{http://www.ams.org/mathscinet-getitem?mr=#1}{#2}
}
\providecommand{\href}[2]{#2}
\begin{thebibliography}{FBMGVP21}

\bibitem[Aze24]{azemar2021qualitative}
A.~Azemar, \emph{A qualitative description of the horoboundary of the
  {T}eichm\"uller metric}, Algebr. Geom. Topol. \textbf{24} (2024), no.~7,
  3919--3984. \MR{4840385}

\bibitem[BH99]{BridsonHaefliger}
M.~R. Bridson and A.~Haefliger, \emph{Metric spaces of non-positive curvature},
  Grundlehren der mathematischen Wissenschaften [Fundamental Principles of
  Mathematical Sciences], vol. 319, Springer-Verlag, Berlin, 1999. \MR{1744486}

\bibitem[CCF{\etalchar{+}}18]{toy}
Y.~Chen, R.~Chernov, M.~Flores, M.~Fortier~Bourque, S.~Lee, and B.~Yang,
  \emph{Toy {T}eichm\"uller spaces of real dimension 2: the pentagon and the
  punctured triangle}, Geom. Dedicata \textbf{197} (2018), 193--227.
  \MR{3876303}

\bibitem[FB23]{fortier2019Divergent}
M.~Fortier~Bourque, \emph{A divergent horocycle in the horofunction
  compactification of the {T}eichm\"uller metric}, Ann. Inst. Fourier
  (Grenoble) \textbf{73} (2023), no.~5, 1885--1902. \MR{4655380}

\bibitem[FBMGVP21]{bolza}
M.~Fortier~Bourque, D.~Martinez-Granado, and F.~Vargas~Pallette, \emph{The
  extremal length systole of the {B}olza surface}, preprint,
  \href{https://arxiv.org/abs/2105.03871}{arXiv:2105.03871}, 2021.

\bibitem[FBR18]{FBRafi}
M.~Fortier~Bourque and K.~Rafi, \emph{Non-convex balls in the {T}eichm\"uller
  metric}, J. Differential Geom. \textbf{110} (2018), no.~3, 379--412.
  \MR{3880229}

\bibitem[FLP12]{FLP}
A.~Fathi, F.~Laudenbach, and V.~Po\'enaru, \emph{Thurston's work on surfaces},
  Mathematical Notes, vol.~48, Princeton University Press, Princeton, NJ, 2012,
  Translated from the 1979 French original by Djun M. Kim and Dan Margalit.
  \MR{3053012}

\bibitem[FM12]{FarbMargalit}
B.~Farb and D.~Margalit, \emph{A primer on mapping class groups}, Princeton
  Mathematical Series, vol.~49, Princeton University Press, Princeton, NJ,
  2012. \MR{2850125}

\bibitem[GM91]{gardmasur1991extrgeom}
F.~P. Gardiner and H.~Masur, \emph{Extremal length geometry of
  {T}eichm\"{u}ller space}, Complex Variables Theory Appl. \textbf{16} (1991),
  no.~2-3, 209--237. \MR{1099913}

\bibitem[GM20]{GekhtmanMarkovic}
D.~Gekhtman and V.~Markovic, \emph{Classifying complex geodesics for the
  {C}arath\'eodory metric on low-dimensional {T}eichm\"uller spaces}, J. Anal.
  Math. \textbf{140} (2020), no.~2, 669--694. \MR{4093921}

\bibitem[Gro81]{Gromov}
M.~Gromov, \emph{Hyperbolic manifolds, groups and actions}, Riemann surfaces
  and related topics: {P}roceedings of the 1978 {S}tony {B}rook {C}onference
  ({S}tate {U}niv. {N}ew {Y}ork, {S}tony {B}rook, {N}.{Y}., 1978), Ann. of
  Math. Stud., vol. No. 97, Princeton Univ. Press, Princeton, NJ, 1981,
  pp.~183--213. \MR{624814}

\bibitem[HM79]{HubbardMasur}
J.~Hubbard and H.~Masur, \emph{Quadratic differentials and foliations}, Acta
  Math. \textbf{142} (1979), no.~3-4, 221--274. \MR{523212}

\bibitem[HT22]{HatcherThurston}
A.~Hatcher and W.~Thurston, \emph{A presentation for the mapping class group of
  a closed orientable surface}, Collected works of {W}illiam {P}. {T}hurston
  with commentary. {V}ol. {I}. {F}oliations, surfaces and differential
  geometry, Amer. Math. Soc., Providence, RI, [2022] \copyright2022, Reprint of
  [0579573], pp.~457--473. \MR{4554450}

\bibitem[Hub06]{Hubbard}
J.~H. Hubbard, \emph{Teichm\"uller theory and applications to geometry,
  topology, and dynamics. {V}ol. 1}, Matrix Editions, Ithaca, NY, 2006,
  Teichm\"uller theory, With contributions by Adrien Douady, William Dunbar,
  Roland Roeder, Sylvain Bonnot, David Brown, Allen Hatcher, Chris Hruska and
  Sudeb Mitra, With forewords by William Thurston and Clifford Earle.
  \MR{2245223}

\bibitem[Ker80]{Kerckhoff}
S.~P. Kerckhoff, \emph{The asymptotic geometry of {T}eichm\"uller space},
  Topology \textbf{19} (1980), no.~1, 23--41. \MR{559474}

\bibitem[KMN06]{KarlssonEtAl}
A.~Karlsson, V.~Metz, and G.~A. Noskov, \emph{Horoballs in simplices and
  {M}inkowski spaces}, Int. J. Math. Math. Sci. (2006), Art. ID 23656, 20.
  \MR{2268510}

\bibitem[KN09]{KleinNicas1}
T.~Klein and A.~Nicas, \emph{The horofunction boundary of the {H}eisenberg
  group}, Pacific J. Math. \textbf{242} (2009), no.~2, 299--310. \MR{2546714}

\bibitem[KN10]{KleinNicas2}
\bysame, \emph{The horofunction boundary of the {H}eisenberg group: the
  {C}arnot-{C}arath\'eodory metric}, Conform. Geom. Dyn. \textbf{14} (2010),
  269--295. \MR{2738530}

\bibitem[LM10]{LenzhenMasur}
A.~Lenzhen and H.~Masur, \emph{Criteria for the divergence of pairs of
  {T}eichm\"uller geodesics}, Geom. Dedicata \textbf{144} (2010), 191--210.
  \MR{2580426}

\bibitem[LS14]{LiuSu}
L.~Liu and W.~Su, \emph{The horofunction compactification of the
  {T}eichm\"uller metric}, Handbook of {T}eichm\"uller theory. {V}ol. {IV},
  IRMA Lect. Math. Theor. Phys., vol.~19, Eur. Math. Soc., Z\"urich, 2014,
  pp.~355--374. \MR{3289706}

\bibitem[Mas75]{Masur}
H.~Masur, \emph{On a class of geodesics in {T}eichm\"uller space}, Ann. of
  Math. (2) \textbf{102} (1975), no.~2, 205--221. \MR{385173}

\bibitem[Min96]{Minsky}
Y.~N. Minsky, \emph{Extremal length estimates and product regions in
  {T}eichm\"uller space}, Duke Math. J. \textbf{83} (1996), no.~2, 249--286.
  \MR{1390649}

\bibitem[Miy14]{Miyachi}
H.~Miyachi, \emph{Extremal length boundary of the {T}eichm\"uller space
  contains non-{B}usemann points}, Trans. Amer. Math. Soc. \textbf{366} (2014),
  no.~10, 5409--5430. \MR{3240928}

\bibitem[MW95]{MasurWolf}
H.~A. Masur and M.~Wolf, \emph{Teichm\"uller space is not {G}romov hyperbolic},
  Ann. Acad. Sci. Fenn. Ser. A I Math. \textbf{20} (1995), no.~2, 259--267.
  \MR{1346811}

\bibitem[Rie02]{Rieffel}
M.~A. Rieffel, \emph{Group {$C^*$}-algebras as compact quantum metric spaces},
  Doc. Math. \textbf{7} (2002), 605--651. \MR{2015055}

\bibitem[Str66]{strebel1966modular}
K.~Strebel, \emph{\"{U}ber quadratische {D}ifferentiale mit geschlossenen
  {T}rajektorien und extremale quasikonforme {A}bbildungen}, Festband 70.
  {G}eburtstag {R}. {N}evanlinna, Springer, Berlin-New York, 1966,
  pp.~105--127. \MR{209470}

\bibitem[Wal07]{WalshMinkowski}
C.~Walsh, \emph{The horofunction boundary of finite-dimensional normed spaces},
  Math. Proc. Cambridge Philos. Soc. \textbf{142} (2007), no.~3, 497--507.
  \MR{2329698}

\bibitem[Wal08]{WalshHilbert}
\bysame, \emph{The horofunction boundary of the {H}ilbert geometry}, Adv. Geom.
  \textbf{8} (2008), no.~4, 503--529. \MR{2456635}

\bibitem[Wal14]{WalshThurston}
\bysame, \emph{The horoboundary and isometry group of {T}hurston's {L}ipschitz
  metric}, Handbook of {T}eichm\"uller theory. {V}ol. {IV}, IRMA Lect. Math.
  Theor. Phys., vol.~19, Eur. Math. Soc., Z\"urich, 2014, pp.~327--353.
  \MR{3289705}

\bibitem[Wal19]{walsh2019asympgeom}
\bysame, \emph{The asymptotic geometry of the {T}eichm\"{u}ller metric}, Geom.
  Dedicata \textbf{200} (2019), 115--152. \MR{3956189}

\end{thebibliography}

\end{document}